\newif\ifcc\IfFileExists{cc.cls}{\cctrue}{\ccfalse}\ifcc

\documentclass{cc}

\usepackage{url}

\newcommand{\pb}[1]{\textsc{#1}} 
\newenvironment{Problems}{\begin{list}{}{\setlength{\itemsep}{10pt}
\setlength{\topsep}{10pt}
} \it}{\end{list}}

\renewcommand{\eqref}{\ref}

\newcommand{\compc}[1]{\textsf{#1}}
\newcommand{\Card}{\textrm{Card}}
\newcommand{\cut}[1]{\overline{#1}} 
\newcommand{\Rp}{\leq_{\textit{par}}}
\newcommand{\stretchg}{\widetilde{g}}
\newcommand{\ZZ}{\mathbb{Z}}
\newcommand{\NN}{\mathbb{N}}
\renewcommand{\S}{\mathfrak{S}}
\newcommand{\CC}{\mathbb{C}}
\newcommand{\RR}{\mathbb{R}}
\newcommand{\f}[2]{{\frac {#1} {#2}}}
\newcommand{\bitlength}[1]{\langle #1 \rangle}
\newcommand{\poly}{{\mbox {poly}}}

\contact{mrosas@us.es}
\title{Reduced Kronecker Coefficients\\
 and\\
 counter--examples to\\
 Mulmuley's\\
strong saturation conjecture {\bf SH}\\
\footnotesize With an appendix by Ketan Mulmuley} 
\titlehead{Reduced Kronecker coefficients}
\author{Emmanuel Briand\\
Departamento de \'Algebra,\\
Facultad de Matem\'aticas,\\
Aptdo. de Correos 1160,\\
41080 Sevilla, Spain.\\
\email{ebriand@us.es}
\and
Rosa Orellana\\
Dartmouth College,\\
 Mathematics Department,\\
6188 Kemeny Hall,\\
Hanover, NH 03755, USA.\\
\email{Rosa.C.Orellana@dartmouth.edu}
\and
Mercedes Rosas\\
Departamento de \'Algebra,\\
Facultad de Matem\'aticas,\\
Aptdo. de Correos 1160,\\
41080 Sevilla, Spain.\\
\email{mrosas@us.es}
}

\begin{abstract}
We provide counter--examples to Mulmuley's strong saturation conjecture ({\bf strong SH}) for the Kronecker coefficients. This conjecture was proposed in the setting of Geometric Complexity Theory to show that deciding whether or not a Kronecker coefficient is zero can be done in polynomial time. We also provide a short proof of the \#P--hardness of computing the Kronecker coefficients. Both results rely on the connections between the Kronecker coefficients and another family of structural constants in the representation theory of the symmetric groups, Murnaghan's reduced Kronecker coefficients.

An appendix by Mulmuley introduces a relaxed form of the saturation hypothesis {\bf SH}, still strong enough for the aims of Geometric Complexity Theory.
\end{abstract}


\begin{subject}
05E05; 
03D15; 
20G05.
\end{subject}

\begin{document}

\defcitealias{GCT1}{GCT1:SIAM}
\defcitealias{GCT2}{GCT2:SIAM}
\defcitealias{GCT3}{GCT3}
\defcitealias{GCT4}{GCT4}
\defcitealias{GCT6}{GCT6}
\defcitealias{GCT7}{GCT7}
\defcitealias{GCT8}{GCT8}

\section{Introduction}

A major unsolved problem in group representation theory is raised by the 
tensor product of irreducible representations of a symmetric group $\S_n$: Provide a (positive) combinatorial interpretation for the multiplicities of its decomposition into irreducible representations (the \emph{Kronecker coefficients} $g_{\mu,\nu}^{\lambda}$). Recently, the Kronecker coefficients started being examined also under the angle of computational complexity. The problems considered are:\begin{Problems}
\item \pb{Kron}: (Computation problem) Given labels $\lambda$, $\mu$, $\nu$, 
compute the Kronecker coefficient $g_{\mu,\nu}^{\lambda}$.
\item \pb{ZeroKron}: (Decision problem) Given labels $\lambda$, $\mu$ and $\nu$, decide whether $g_{\mu,\nu}^{\lambda}$ is positive or zero.
\end{Problems}
\noindent The labels $\lambda$, $\mu$, $\nu$ of the Kronecker coefficients are integer partitions (finite, nonincreasing sequences of positive integers) and the complexity of these problems is measured with respect to their bitlength.

A family of related coefficients, much better understood, are the \emph{Littlewood--Richardson coefficients} $c_{\mu,\nu}^{\lambda}$ (the analogues of the Kronecker coefficients for representations of the complex linear groups $GL_n(\CC)$). \citet{Narayanan} showed that computing the Littlewood--Richardson coefficients is a \compc{\#P}--complete problem, which implies that there is no polynomial--time algorithm to perform this task, unless $\compc{P}=\compc{NP}$. On the other hand, deciding the positivity of the Littlewood--Richardson coefficients can be done in polynomial time \citep{GCT3,Knutson:Tao:Notices,DeLoera:McAllister}, as a consequence of the \emph{saturation property} established by \citet{Knutson:Tao:1}:
\begin{equation}\label{eq:satLR}
c_{\mu,\nu}^{\lambda} >0  \qquad \Leftrightarrow \qquad c_{N\mu,N\nu}^{N\lambda}>0 \; \textit{ for all positive integers $N$,}
\end{equation}
where $N \lambda$ is just the partition $(N \lambda_1,N \lambda_2, \ldots, N\lambda_k)$ if $\lambda=(\lambda_1,\lambda_2, \ldots, \lambda_k)$, and likewise for $N \mu$ and $N \nu$.

The problem of determining whether or not \pb{Kron} is also in \compc{\#P} is open \citep[Question 2.1 in][]{GCT0}. \citet{Burgisser:Ikenmeyer} showed  that \pb{Kron} is \compc{\#P}--hard, and \compc{GapP}--complete.

That \pb{ZeroKron} is in \compc{P} is an open conjecture that lies at the heart of a detailed plan, \emph{Geometric Complexity Theory}, that Ketan Mulmuley and Milind Sohoni elaborated to prove that $\compc{P} \neq \compc{NP}$ over the complex numbers: an arithmetic, non--uniform version of $\compc{P} \neq \compc{NP}$ \citep[See the series of papers][referred to as GCT1--8 below]{GCT1:SIAM,GCT2:SIAM,GCT3,GCT4,GCT5,GCT6,GCT7,GCT8}. Indeed, Mulmuley proposed in \citetalias{GCT6} a  variant of the strategy used to show that the positivity of the Littlewood--Richardson coefficients can be decided in polynomial time, that would imply that \pb{ZeroKron} is also in $\compc{P}$. 

Mulmuley conjectured in \citetalias{GCT6} (versions 1--3) that the following  two hypotheses hold: 

(i) A \emph{positivity hypothesis} {\bf PH1}, that says that the Kronecker coefficients $g_{N\mu,N\nu}^{N\lambda}$ for $N \in \NN^*$ count the integral points in the dilations $N \mathcal{P}$ of a polytope $\mathcal{P}=\mathcal{P}(\lambda,\mu,\nu)$ whose non--emptiness can be decided in polynomial time. (Note  that {\bf PH1}  implies that \pb{Kron} is in \compc{\#P}).

(ii) A \emph{saturation hypothesis} {\bf SH}, that reduces the test that $\mathcal{P}(\lambda,\mu,\nu)$ contains an integral point (integer programming) to a test of non--emptiness (linear programming). Since it is known that the verbatim translation of the saturation property of the Littlewood--Richardson coefficients \eqref{eq:satLR} does not  hold for the Kronecker coefficients, Mulmuley proposed a variant of it, see \ref{sec:preliminaries},  and conjectured that it holds for the Kronecker coefficients ({\bf SH}). 

Together, hypotheses {\bf SH} and {\bf PH1} imply that \pb{ZeroKron} is in \compc{P} \citepalias[See Theorem 1.4.1 in][]{GCT6}.

The main result of this paper shows that the situation is more complicated than expected: {\bf SH}, as formulated in \citetalias[][versions 1--3]{GCT6}, does not hold for the Kronecker coefficients.  \ref{thm:main}
 provides infinitely many counter--examples. They belong to the family of the Kronecker coefficients $g_{\mu,\nu}^{\lambda}$ where $\mu$ and $\nu$ have at most $2$ parts and $\lambda$ has at most $3$ parts. They
are, actually, \emph{all} the counter--examples in this family (see \ref{subsec:get cex}).  

After a first version of this paper was published on the preprint server ArXiv,
\citet{Mulmuley:erratum} proposed a new, weaker saturation conjecture {\bf SH}, still strong enough for the aims of Geometric Complexity Theory. This correction is appended to the present paper. In what follows, we will refer to the disproved saturation hypothesis, stated in \citetalias[][versions 1--3]{GCT6}, as the \emph{strong saturation hypothesis, {\bf strong SH}}.

This work is organized as follows.
\ref{thm:main} is formally proved in \ref{sec:proof main} after some preliminaries (\ref{sec:preliminaries}). In \ref{sec:explain} we explain how the counter--examples were obtained: we were able to check exhaustively {\bf strong SH} 
for all Kronecker coefficients indexed by two two--row shapes (the Kronecker coefficients $g_{\mu,\nu}^{\lambda}$ where $\mu$ and $\nu$ have at most two terms) thanks to new explicit formulas for them. Previously known formulas,  due to \citet{Remmel:Whitehead} and \citet{Rosas:2001} were not suitable for this study. The new formulas were obtained by considering another family of structural constants related to representations of the symmetric groups introduced long time ago by \citet{Murnaghan:1938} and called \emph{Reduced Kronecker coefficients} by \citet{Klyachko}. They are defined precisely in \ref{subsec:reduced}. In \ref{sec:relevance} we point out the possible relevance of the reduced Kronecker coefficients in the complexity issues about Kronecker coefficients. As an illustration,  a very short and simple proof of the \compc{\#P}--hardness of \pb{Kron} is presented. 
The paper ends with Mulmuley's appendix proposing a relaxed version of the hypothesis {\bf strong SH}.


\section{Preliminaries and Main Result}\label{sec:preliminaries}
\subsection{Preliminaries on Kronecker coefficients.}\label{subsec:rep}

A \emph{partition} is a finite nonincreasing sequence of positive integers $\lambda=(\lambda_1,\lambda_2,\ldots,\lambda_k)$. We allow ourselves, when convenient, to represent also $\lambda$ by the sequences obtained by appending trailing zeros:  $\lambda = (\lambda_1,\lambda_2,\ldots,\lambda_k,0,\ldots,0)$. The (non--zero) terms $\lambda_i$ of $\lambda$ are usually called its \emph{parts}. The \emph{length} (number of parts) $k$ is denoted with $\ell(\lambda)$. The sum of the parts of $\lambda$ is called \emph{weight of $\lambda$} and denoted with $|\lambda|$.

For $N$ and $d$ positive integers, the partition $(N,N,\ldots,N)$ having $d$ parts all equal to $N$ is denoted with $(N^d)$. Partitions can be added: $\lambda+\mu=(\lambda_1+\mu_1,\lambda_2+\mu_2, \ldots)$ and stretched: for $N$ positive integer, $N\lambda=(N \lambda_1, N \lambda_2,\ldots)$. 

The irreducible (finite--dimensional, complex) representations $V_{\lambda}(\S_n)$ of the symmetric group $\S_n$ are indexed by the partitions $\lambda$ of weight $n$.  Given two irreducible representations $V_{\mu}(\S_n)$ and $V_{\nu}(\S_n)$, one can form their tensor product and decompose it into irreducible representations. Such a  decomposition takes the form:
\[
V_{\mu}(\S_n) \otimes V_{\nu}(\S_n) \cong \bigoplus_{\lambda\text{ s.t. } |\lambda|=n} g_{\mu,\nu}^{\lambda} V_{\lambda}(\S_n)
\] 
where $g_{\mu,\nu}^{\lambda} V_{\lambda}(\S_n)$ means ``the direct sum of $g_{\mu,\nu}^{\lambda}$ copies of $V_{\lambda}(\S_n)$''. Each multiplicity $g_{\mu,\nu}^{\lambda}$ is uniquely determined by $\lambda$, $\mu$, $\nu$. By considering all symmetric groups $\S_n$ this defines an infinite family of integers $g_{\mu,\nu}^{\lambda}$ indexed by triples of partitions fulfilling $|\lambda|=|\mu|=|\nu|$. They are the Kronecker coefficients.

The Kronecker coefficients can also be interpreted in the setting of representations of the general linear group (this is, actually, the relevant interpretation in Geometric Complexity Theory). The irreducible (finite--dimensional, polynomial) representations $V_{\lambda}(GL_k(\CC))$ of $GL_k(\CC)$ are indexed by the partitions $\lambda$ with length at most $k$. The irreducible representations of the Cartesian product $GL_m(\CC) \times GL_n(\CC)$ are just the tensor products $V_{\mu}(GL_m(\CC)) \otimes V_{\nu}(GL_n(\CC))$. The Kronecker coefficient $g_{\mu,\nu}^{\lambda}$ is the multiplicity of $V_{\mu}(GL_m(\CC)) \otimes V_{\nu}(GL_n(\CC))$ in $V_{\lambda}(GL_{mn}(\CC))$, seen as a representation of  $GL_m(\CC) \times GL_n(\CC)$ through the Kronecker product of matrices. 
%
This follows from Schur--Weyl duality \citep[see][]{Kraft:Procesi}. A simple consequence is that $g_{\mu,\nu}^{\lambda} $ can be positive only if $\ell(\lambda) \leq \ell(\mu) \ell(\nu)$. Another consequence is provided by the following lemma. 
\begin{lemma}
Let $\mu$, $\nu$, $\lambda$ be partitions with $\ell(\mu) \leq m$, $\ell(\nu) \leq n$ and $\ell(\lambda) \leq mn$. Below we write $\lambda$ as a sequence of exactly $mn$ terms, by appending trailing zeros if necessary: $\lambda=(\lambda_1,\lambda_2,\ldots, \lambda_{mn-1},\lambda_{mn})$,
 and likewise for $\mu$ ($m$ terms) and $\nu$ ($n$ terms). Set $k=\lambda_{mn}$, $k_1=m k$ and $k_2=n k$. If $\mu_m < k_2$ or $\nu_n < k_1$ then $g_{\mu,\nu}^{\lambda}=0$. Else, 
\begin{equation}\label{eq:simplification}
g_{\mu,\nu}^{\lambda}=
g_{(\mu_1-k_2,\mu_2-k_2,\ldots,\mu_m-k_2)(\nu_1-k_1,\nu_2-k_1,\ldots,\nu_n-k_1)}^{(\lambda_1-k,\lambda_2-k,\ldots,\lambda_{mn-1}-k,0)}.
\end{equation}
\end{lemma}

\begin{proof}
Let $m$, $n$ be positive integers and
let $\lambda$ be a partition with length at most $mn$. As representations of $GL(\CC^m) \times
GL(\CC^n)$:
\[
V_{\lambda}\left(GL(\CC^m \otimes \CC^n)\right)
\cong
\mathop{\bigoplus_{\ell(\mu) \leq m}}_{\ell(\nu) \leq n}
g^{\lambda}_{\mu,\nu} \, V_{\mu}\left(GL(\CC^m)\right)\otimes V_{\nu}\left(GL(\CC^n)\right).
\]
For any vector space $U$ of dimension $d$, let $\det U=V_{(1^d)}\left(GL(U)\right)$ be
its \emph{determinant representation} (remember that $(1^d)$ stands for the partition $(1,1,\ldots,1)$ with $d$ parts, all equal to $1$). That is, the one dimensional
representation of $GL(U)$ such that $g(u)=\det(g) u$ for all $g \in
GL(U)$. We have $V_{\lambda+(1^{d})}\left(GL(U)\right) \cong \det(U) \otimes V_{\lambda}\left(GL(U)\right)$. 
This fact, together with the isomorphism:
\[
\det(\CC^m \otimes \CC^n) \cong
\det(\CC^m)^{\otimes n} \otimes \det(\CC^n)^{\otimes m},
\]
implies the following decomposition:
\[
V_{\lambda+(1^{mn})}\left(GL(\CC^m \otimes \CC^n)\right)
\cong
\mathop{\bigoplus_{\ell(\mu) \leq m}}_{\ell(\nu) \leq n}
g^{\lambda}_{\mu,\nu} \, V_{\mu+(n^m)}\left(GL(\CC^m)\right)\otimes V_{\nu+(m^n)}\left(GL(\CC^n)\right).
\]
From this we obtain the invariance relation:
\[
g_{\mu+(n^m),\nu+(m^n)}^{\lambda+(1^{mn})}
=g_{\mu,\nu}^{\lambda}.
\]
Moreover, we easily see
that $g_{\mu\nu}^{\lambda}$ can be positive only when $\mu_m \geq n
\lambda_{mn}=k_2$ and $\nu_n \geq m \lambda_{mn}=k_1$.
\end{proof}

The simplest non--trivial, meaningful family of Kronecker coefficients $g_{\mu,\nu}^{\lambda}$ closed under stretching (that is, under the map that sends the triple  $\lambda$, $\mu$, $\nu$ into $N \lambda$, $N\mu$, $N\nu$) is the family of coefficients $g_{(\mu_1,\mu_2)(\nu_1,\nu_2)}^{(\lambda_1,\lambda_2,\lambda_3,\lambda_4)}$. But in view of \eqref{eq:simplification} its study can be reduced to the study of the coefficients $g_{(\mu_1,\mu_2)(\nu_1,\nu_2)}^{(\lambda_1,\lambda_2,\lambda_3)}$.


\subsection{Reduced Kronecker Coefficients}\label{subsec:reduced}

Let us introduce another family of constants, closely related to the Kronecker coefficients, which will play an important role in this paper.

For any partition $\lambda=(\lambda_1,\lambda_2,\ldots,\lambda_k)$ and integer $n$, denote with $(n-|\lambda|,\lambda)$ the sequence obtained by prepending $n-|\lambda|$ to $\lambda$, \emph{i.e.} $(n-|\lambda|,\lambda)=(n-|\lambda|,\lambda_1,\lambda_2,\ldots,\lambda_k)$. This is a partition when $n \geq |\lambda|+\lambda_1$. \citet{Murnaghan:1938} established that for any three partitions $\alpha$, $\beta$, $\gamma$, the sequence with general term $g_{(n-|\alpha|,\alpha)(n-|\beta|,\beta)}^{(n-|\gamma|,\gamma)}$ is stationary. Let $\overline{g}_{\alpha,\beta}^{\gamma}$ be its limit. Following \citet{Klyachko} we call these constants $\overline{g}_{\alpha,\beta}^{\gamma}$ the \emph{reduced Kronecker coefficients}. 

Upper bounds for the index when the sequences with general term 
\[
g_{(n-|\alpha|,\alpha)(n-|\beta|,\beta)}^{(n-|\gamma|,\gamma)}
\]
become stationary are discussed by \citet{Brion:Foulkes,Vallejo,Briand:Orellana:Rosas:stability}.


\subsection{Saturation Hypotheses and counter--examples to the hypothesis {\bf strong SH}}

The verbatim translation of the saturation property \eqref{eq:satLR} that holds for the Littlewood--Richardson coefficients is known not to hold for the Kronecker coefficients. The simplest counter--example may be $g_{(1,1)(1,1)}^{(1,1)}=0$ but  $g_{(2,2)(2,2)}^{(2,2)}=1$.  Indeed, 
\[
g_{(N,N),(N,N)}^{(N,N)}=
\left\lbrace
\begin{matrix}
0 & \textit{ for odd $N$,}\\
1 & \textit{ for even $N$.}
\end{matrix}
\right.
\]
Many more such counter--examples exist.

To present an alternative saturation property we consider the \emph{stretching function} $\stretchg_{\mu,\nu}^{\lambda}$ (where $\mu$, $\nu$, $\lambda$ are partitions of the same weight) defined by:
\[
\stretchg_{\mu,\nu}^{\lambda}: 
\quad N\in \NN^* \longmapsto g_{N\mu,N\nu}^{N\lambda}.
\]
Mulmuley proved \citepalias[][Theorem 1.6.1.b]{GCT6} that $\stretchg_{\mu,\nu}^{\lambda}$ is always \emph{a quasipolynomial}, \emph{i.e.}  a function of the form:
\begin{equation}\label{eq:quasipol}
F: \quad N \in \NN^* \longmapsto \left\lbrace
\begin{matrix}
F_1(N) & \text{ if $N \equiv 1 \mod k$,} \\
F_2(N) & \text{ if $N \equiv 2 \mod k$,} \\
\vdots & \vdots \\
F_k(N) & \text{ if $N \equiv k \mod k$} 
\end{matrix}
\right.
\end{equation}
where $k$ is a positive integer and $F_1$, $F_2$, \ldots, $F_k$ are polynomials \citep[see][Section 4.4 as a reference on quasipolynomials]{Stanley:vol1}. A quasipolynomial is said to be \emph{saturated} \citepalias[][def. 1.2.4]{GCT6} when it has the property: 
\[
F(1) = 0 \Rightarrow F_1 \textit{ vanishes identically.}
\]
We are now ready to state Mulmuley's strong saturation hypothesis  {\bf strong SH} for the Kronecker coefficients \citepalias[][versions 1--3]{GCT6}:

\begin{Problems}
\item {\bf Strong SH:} {\itshape The stretching quasipolynomials $\stretchg_{\mu,\nu}^{\lambda}$ of the Kronecker coefficients are saturated.}
\end{Problems}

Observe that {\bf Strong SH} implies that for any three partitions $\lambda$, $\mu$, $\nu$ of the same weight,
\[
g_{\mu,\nu}^{\lambda}=0 \Rightarrow g_{N\mu,N\nu}^{N\lambda}=0 \;\textit{ for infinitely many positive integers $N$}.
\]
Hence, the following theorem provides an infinite family of counter--examples to {\bf Strong SH}. (Its proof is provided in the next section.)
\begin{theorem}\label{thm:main}
Let $i$, $j$, $k$ be integers such that $i>j>0$ and $k > 2i+j$. 
Let 
\[
\alpha=(k,k),\qquad 
\beta=(k+1,k-1),\qquad
\gamma=(2k-2i-2j,2i,2j)
\]
Then:
\begin{equation}\label{cex}
g_{N\alpha,N\beta}^{N\gamma}=
N/2+
\left\lbrace
\begin{matrix}
1 & \textit{ for even $N$,}\\
-1/2 & \textit{ for odd $N$.}
\end{matrix}
\right.
\end{equation}
In particular, $g_{\alpha,\beta}^{\gamma}=0$ and $g_{N\alpha,N\beta}^{N\gamma}>0$ for all $N>1$.
\end{theorem}
The smallest counter--example in this family is $\stretchg_{(6,6)(7,5)}^{(6,4,2)}$.

\begin{remark}\label{rem:codim} \emph{All} counter--examples $\stretchg^{\lambda}_{\mu,\nu}$ to {\bf Strong SH}  with $\ell(\mu) \leq 2$, $\ell(\nu)\leq 2,$ and $\ell(\lambda) \leq 3$ 
are given by  \ref{thm:main}, up to permutation of $\mu$ and $\nu$. This follows from the calculations in \citet{Briand:Orellana:Rosas:FPSAC,Briand:Orellana:Rosas:Chamber} reported in \ref{subsec:get cex}.
All triples $(\mu,\nu,\lambda)$ corresponding to these counter--examples are contained in a codimension $2$ affine subspace of the space of parameters, namely the subspace defined by $\mu_2=\nu_2+1$, $\mu_1=\mu_2$. 
In addition, they  shall fulfill congruences: $\lambda_2 \equiv \lambda_3 \equiv 0 \mod 2$.
Under these conditions, it is not surprising that these counter--examples escaped the sampling taken in \citetalias{GCT6}, Section 6.2, as an experimental verification of {\bf Strong SH}.
\end{remark}

\begin{remark}
Strikingly, the family of Kronecker coefficients $g_{\mu,\nu}^{\lambda}$ with $\mu=(k,k)$, which comprises the counter--examples of \ref{thm:main}, 
 has been independently studied in the recent preprint by \citet{Brown:VanWilli:Zabrocki} motivated by questions in mathematical physics.
\end{remark}


\subsection{The Positivity Hypothesis {\bf PH2}, and counter--examples}\label{subsec:PH2}

\hyphenation{quasi-poly-no-mial}

A quasipolynomial $F$  is said to be \emph{positive} when the coefficients of all polynomials $F_i$ in \eqref{eq:quasipol} are nonnegative  \citepalias[][def. 1.2.2]{GCT6}.  \citet{GCT6}  conjectured:

\begin{Problems}
\item {\bf Strong PH2:} {\itshape The stretching quasipolynomials $\stretchg_{\mu,\nu}^{\lambda}$ of the Kronecker coefficients are positive.}
\end{Problems}

This conjecture was first proposed for the Littlewood--Richardson coefficients by \citet{King:Tollu:Toumazet} 
and later for the generalized Littlewood--Richardson coefficients associated to the classical Lie groups of type $B$, $C$ and $D$ by \citet{DeLoera:McAllister}. The conjecture is still open for both cases.

Note that {\bf Strong PH2} implies straightforwardly {\bf Strong SH}. The counter--examples to {\bf Strong SH} provided by \ref{thm:main} are thus also counter--examples to {\bf Strong PH2} for Kronecker coefficients. The following is a counter--example to {\bf Strong PH2} that nevertheless fulfills {\bf Strong SH}:
\[
g^{(10N,6N,2N)}_{(10N,8N)(11N,7N)}=
7/4 \; N^2 + 3/2\; N +
\left\lbrace
\begin{matrix}
1    &\textit{ for even $N$,}\\
-1/4 &\textit{ for odd $N$.}
\end{matrix}
\right.
\]
There exist many more counter--examples to {\bf Strong PH2} than to {\bf Strong SH}. Indeed, there are several full--dimensional rational convex polyhedral cones $\sigma_i$ of $\RR^5$ such that 
for all $(\lambda_1,\lambda_2,\lambda_3,\mu_2,\mu_3) \in \sigma_i \cap \ZZ^5$ fulfilling $\lambda_1 \equiv \lambda_2 \equiv \lambda_3 \equiv \mu_2+\nu_2+1 \equiv 0 \mod 2$, the quasipolynomial $\stretchg_{(\mu_1,\mu_2)(\nu_1,\nu_2)}^{(\lambda_1,\lambda_2,\lambda_3)}$ is not positive. (See \ref{subsec:get cex}.)

\begin{remark} Although the Littlewood--Richardson coefficients are particular cases of Kronecker coefficients (see \ref{sec:relevance}), no counter--example to {\bf Strong PH2} for the Littlewood--Richardson coefficients was obtained this way.
\end{remark}

\subsection{Relaxed saturation hypothesis SH and relaxed positivity hypothesis PH2}\label{subsec:relaxed}

After a first version of this paper was released on the preprint server ArXiv, \citet{Mulmuley:erratum} 
gave new, weaker versions of the saturation hypothesis {\bf SH} and of the positivity hypothesis {\bf PH2}. These hypotheses are still strong enough for the aims of Geometric Complexity Theory. Mulmuley's note \citep{Mulmuley:erratum} is appended to this paper. 

The relaxed hypotheses {\bf PH2} and {\bf SH} hold for the Kronecker coefficients indexed by two two--rows shapes. In particular, {\bf SH} a) follows from \ref{rem:codim} in this case. For this family of Kronecker coefficients, the positivity index (see the appendix) is at most 1 (see \ref{rem:positivity index}).


\section{Proof of \ref{thm:main}}\label{sec:proof main}

The proof of \ref{thm:main} will rely on the following lemma, which is a slight simplification of \citealp[Theorem 1 in][]{Rosas:2001}.
\begin{lemma}\label{lemma:rosas}
Let $n$ be a natural integer. Let $\lambda$, $\mu$, $\nu$ be three partitions of $n$. Suppose that $\mu$ and $\nu$ have at most two parts and $\lambda$ has at most three parts. Suppose additionally that $\mu_2 \geq \nu_2$. 
Then:
\begin{equation}\label{frosas}
g_{\mu,\nu}^{\lambda}=
\Card \left(\mathcal{R}_+ \cap \mathcal{Z} \cap \mathcal{L}\right)
-
\Card \left(\mathcal{R}_- \cap \mathcal{Z} \cap \mathcal{L}\right)
\end{equation}
where:
\begin{eqnarray*}
\mathcal{Z}&=&\left\{(x,y) \in \RR^2 \,|\, x+y \leq \mu_2+\nu_2+1 
\text{ and } y-x \geq \mu_2-\nu_2+1
\right\},\\
\mathcal{L}&=&\left\{(x,y) \in \ZZ^2 \,|\, x+y \equiv \mu_2+\nu_2+1 \mod 2 \right\},
\end{eqnarray*}
and $\mathcal{R}_+$ and $\mathcal{R}_-$ are the following two rectangles of $\RR^2$:
\begin{eqnarray*}
\mathcal{R}_+&=&[\lambda_3;\lambda_2] \times [1+\lambda_2;1+\lambda_2+\lambda_3],\\
\mathcal{R}_-&=&[\lambda_3;\lambda_2] \times [2+\lambda_1;2+\lambda_1+\lambda_3].
\end{eqnarray*}
\end{lemma}

\begin{proof}
When $\lambda_1 \geq \lambda_2 + \lambda_3$, 
this is a direct application of Theorem 1 of \citet{Rosas:2001}. When $\lambda_1 < \lambda_2+\lambda_3$, the same theorem states that:
\[
g_{\mu,\nu}^{\lambda}=
\Card \left(\mathcal{R}'_+ \cap \mathcal{Z} \cap \mathcal{L}\right)
-
\Card \left(\mathcal{R}'_- \cap \mathcal{Z} \cap \mathcal{L}\right)
\]
with
\[
\mathcal{R}'_+=[\lambda_3;\lambda_2] \times [1+\lambda_2;1+\lambda_1] \quad \text{ and } \quad
\mathcal{R}'_-=[\lambda_3;\lambda_2] \times [2+\lambda_2+\lambda_3;2+\lambda_1+\lambda_3].
\]
In Formula \eqref{frosas} the points of $\mathcal{Z}$ that lie in $\mathcal{R}_+ \cap \mathcal{R}_-$ have no contribution, thus \eqref{frosas} is equivalent to:\begin{equation}\label{frosas2}
g_{\mu,\nu}^{\lambda}=
\Card\left(\left(\mathcal{R}_+ \setminus \mathcal{R}_-\right) \cap \mathcal{Z} \cap \mathcal{L}\right)
-
\Card\left(\left(\mathcal{R}_- \setminus \mathcal{R}_+\right) \cap \mathcal{Z} \cap \mathcal{L}\right).
\end{equation}
We observe that $\mathcal{R}_+ \setminus \mathcal{R}_-$ has the same integral points as $\mathcal{R}'_+$. Similarly $\mathcal{R}_- \setminus \mathcal{R}_+$ has the same integral points as $\mathcal{R}'_-$. Thus Formula \eqref{frosas2} is true, and so is Formula \eqref{frosas} even when $\lambda_1< \lambda_2+\lambda_3$.
\end{proof}

\begin{namedproof}{Proof of \ref{thm:main}}
Let $n$, $\lambda$, $\mu$ and $\nu$ be as in the hypotheses of \ref{lemma:rosas}. The transformation $(x;y) \mapsto (x;2\mu_2+2-y)$ preserves the sets $\mathcal{Z}$ and $\mathcal{L}$. It also transforms $\mathcal{R}_-$ into $\mathcal{R}''_-=[\lambda_3;\lambda_2] \times [2\mu_2-\lambda_1-\lambda_3;2\mu_2-\lambda_1]$.
As a consequence, 
\[
g_{\mu,\nu}^{\lambda}=
\Card\left(\mathcal{R}_+ \cap \mathcal{Z} \cap \mathcal{L}\right)
-
\Card\left(\mathcal{R}''_- \cap \mathcal{Z} \cap \mathcal{L}\right).
\]
Since the points in the intersection $\mathcal{R}_+ \cap \mathcal{R}''_- $ don't contribute to the above difference,
\[
g_{\mu,\nu}^{\lambda}=
\Card\left( \left(\mathcal{R}_+ \setminus \mathcal{R}''_- \right) \cap \mathcal{Z} \cap \mathcal{L}\right)
-
\Card\left(\left(\mathcal{R}''_- \setminus \mathcal{R}_+ \right) \cap \mathcal{Z} \cap \mathcal{L}\right).
\]
 We apply this to the case when $\mu=N\,\alpha$, $\nu=N\,\beta$ and $\lambda=N\,\gamma$, where $N$ is a positive integer and $\alpha$, $\beta$ and $\gamma$ are as in the theorem. Then
\begin{eqnarray*}
\mathcal{R}_+&=&[2Nj;2Ni] \times [1+2Ni;1+2N(i+j)]\\
\text{ and }\mathcal{R}''_-&=&[2Nj;2Ni] \times [2Ni;2N(i+j)].
\end{eqnarray*}
These two rectangles overlap nearly completely. The integral points of $\mathcal{R}_+ \setminus \mathcal{R}''_-$ are those of the North Side of $\mathcal{R}_+$, which is $\mathcal{S}_+=[2Nj;2Ni] \times \{1 +2N(i+j)\}$.
The integral points of $\mathcal{R}''_- \setminus \mathcal{R}_+$ 
 are those on the South Side of $\mathcal{R}''_-$, which is 
$\mathcal{S}_-=[2Nj;2Ni] \times \{2Ni\}$.
It follows that:
\[
g_{N\alpha,N\beta}^{N\gamma}=
\Card \left(\mathcal{S}_+ \cap \mathcal{Z} \cap \mathcal{L}\right)
-
\Card\left(\mathcal{S}_- \cap \mathcal{Z} \cap \mathcal{L}\right).
\]
The hypotheses $i>j>0$ and  $k > 2i+j$ imply that:
\[
\mathcal{S}_+ \cap \mathcal{Z}=\mathcal{S}_+,\qquad
\mathcal{S}_- \cap \mathcal{Z}=\mathcal{S}'_-=[2Nj;2Ni-N-1] \times \{ 2Ni\}.
\]
We have now:
\[
g_{N\alpha,N\beta}^{N\gamma}=
\Card \left(\mathcal{S}_+ \cap \mathcal{L}\right)
-
\Card\left(\mathcal{S}'_- \cap \mathcal{L}\right).
\]
Observe that for any horizontal segment $\mathcal{S}$ of positive length  with integral end points (since $i>j>0$, the segments $\mathcal{S}_+$ and $\mathcal{S}'_-$ fall in this category), 
\[
\Card \left(\mathcal{S} \cap \mathcal{L}\right)=
\left\lbrace
\begin{array}{ll}
\text{length}(\mathcal{S})/2+1 & \textit{ if both end points are in $\mathcal{L}$,}\\
\text{length}(\mathcal{S})/2+1/2 & \textit{ if exactly one end point is in $\mathcal{L}$,}\\
\text{length}(\mathcal{S})/2 & \textit{ if none of the end points is in $\mathcal{L}$.}
\end{array}
\right.
\]
Since $\mathcal{L}=\{(x,y)\in \ZZ^2\,|\, x+y \equiv N+1 \mod 2\}$, we get
\begin{eqnarray*}
\Card\left(\mathcal{S}_+ \cap \mathcal{L}\right)&=&
\left\lbrace
\begin{array}{ll}
N(i-j)+1 & \textit{ for even $N$,}\\
N(i-j)   &\textit{ for odd $N$,}
\end{array}
\right.
\\
&&\\
\Card\left(\mathcal{S}'_- \cap \mathcal{L}\right)&=&
\left\lbrace
\begin{array}{ll}
N(i-j)-(N+1)/2 +1/2    & \textit{ for even $N$,}\\
N(i-j)-(N+1)/2 +1 &\textit{ for odd $N$.}
\end{array}
\right.
\end{eqnarray*}
Formula \eqref{cex} of \ref{thm:main} follows.
\end{namedproof}


\section{Kronecker coefficients from reduced Kronecker coefficients}\label{sec:explain}

The counter--examples of \ref{thm:main} were found by examining exhaustively the simplest non--trivial family of Kronecker coefficients: the Kronecker coefficients indexed by two two--row shapes, \emph{i.e.} the coefficients $g_{\mu,\nu}^{\lambda}$ where $\mu$ and $\nu$ have at most two parts. By Equation \eqref{eq:simplification}, it is enough to consider those with $\ell(\lambda)\leq 3$,  \emph{i.e.} the coefficients $g_{(\mu_1,\mu_2)(\nu_1,\nu_2)}^{(\lambda_1,\lambda_2,\lambda_3)}$. We give here an account of this investigation, which is presented with more detail by \citet{Briand:Orellana:Rosas:FPSAC,Briand:Orellana:Rosas:Chamber}.

\subsection{Quasipolynomial formulas for $g_{(\mu_1,\mu_2)(\nu_1,\nu_2)}^{(\lambda_1,\lambda_2,\lambda_3)}$}\label{subsec:description}

We obtained a description (call it {\bf D}) of $g_{(\mu_1,\mu_2)(\nu_1,\nu_2)}^{(\lambda_1,\lambda_2,\lambda_3)}$ as a \emph{piecewise quasipolynomial function} of $\lambda_1$, $\lambda_2$, $\lambda_3$, $\mu_2$, $\nu_2$ (note that $\mu_1$ and $\nu_1$ are determined from these parameters by the condition $|\lambda|=|\mu|=|\nu|$).

A \emph{(multivariate) quasipolynomial function} is a function of the form:
\[
F: {\bf x} \in \ZZ^m \longmapsto
\left\lbrace
\begin{matrix}
F_1({\bf x}) & \text{ if ${\bf x} \in C_1$,} \\
F_2({\bf x}) &\text{ if ${\bf x} \in C_2$,} \\
\vdots       & \vdots \\
F_k({\bf x}) &\text{ if ${\bf x} \in C_k$,} \\
\end{matrix}
\right.
\]
where $C_1$, $C_2$, \ldots, $C_k$ are the cosets of a full--rank sublattice of $\ZZ^m$, and $F_1$, $F_2$, \ldots, $F_k$ are polynomials.  

Well--known examples of piecewise quasipolynomial functions are 
the \emph{vector partition functions}. Other examples are provided by the functions of ${\bf y} \in \RR^p$ that count the integral points ${\bf x}$ in polytopes $\mathcal{P}({\bf y}) \subset \RR^q$ defined by a system of linear inequalities of the form
\begin{equation}\label{eq:syst}
f_i({\bf x}) \leq g_i({\bf y})
\end{equation}
where $f_i$ and $g_i$ are integral linear forms. \citealp[See][]{Brion:Vergne} for both families of examples. For these examples, the domains of quasipolynomiality are the maximal cells of a fan (a complex of convex rational polyhedral cones) subdividing a cone of $\RR^n$ (outside the cone the function vanishes).

The explicit formulas obtained by \citet{Remmel:Whitehead} for the coefficients $g_{(\mu_1,\mu_2)(\nu_1,\nu_2)}^{(\lambda_1,\lambda_2,\lambda_3,\lambda_4)}$ are quasipolynomials whose domains are the maximal cells of a polyhedral complex. But the number of domains is huge \citepalias[see][Section 6.2]{GCT6}, which makes these formulas unsuitable for checking conjectures {\bf Strong PH2} and {\bf Strong SH}. Note that in this description, the domains of quasipolynomiality are not cones.

In the description {\bf D} that we obtained, 
the domains of quasipolynomiality are the $74$ maximal cells of a fan 
subdividing a cone $\mathcal{C}$. Note this suggests a stronger form of {\bf PH1}: That $g_{(\mu_1,\mu_2)(\nu_1,\nu_2)}^{(\lambda_1,\lambda_2,\lambda_3)}$ counts the integral points of polytopes $\mathcal{P}(\lambda,\mu,\nu)$ defined by a system of inequalities of the form \eqref{eq:syst}.

\begin{remark} This is the case for the Littlewood--Richardson coefficients: $c_{\mu,\nu}^{\lambda}$ counts the integral points of the Hive polytopes $\mathcal{H}(\lambda,\mu,\nu)$, defined by systems of inequalities of the form \eqref{eq:syst}. \citealp[See][]{Buch}. \citealp[See also][]{Rassart} for the corresponding piecewise quasipolynomial presentation. (It turns out to be piecewise polynomial.)  
\end{remark}

We don't know any family of polytopes whose integral points are counted by $g_{(\mu_1,\mu_2)(\nu_1,\nu_2)}^{(\lambda_1,\lambda_2,\lambda_3)}$. But:
\begin{itemize}
\item[(i)] The reduced Kronecker coefficients (see \ref{subsec:reduced}) indexed by two one--row shapes, have such an interpretation.
\item[(ii)] We are able to recover the Kronecker coefficients from the reduced Kronecker coefficients.
\end{itemize}
This is how we obtained {\bf D}.

We explain (i). 
Assume $\mu_2 \geq \nu_2$. By fixing $\lambda_2$, $\lambda_3$, $\mu_2$, $\mu_3$ and making $|\lambda|=|\mu|=|\nu| \rightarrow \infty$ in  \ref{lemma:rosas}, the rectangle $\mathcal{R}_-$ goes outside $\mathcal{Z}$ while $\mathcal{R}_+$ does not move. It follows that the reduced Kronecker coefficient $\overline{g}_{(\mu_2)(\nu_2)}^{(\lambda_2,\lambda_3)}=\Card\left( \mathcal{R}^+ \cap \mathcal{Z} \cap \mathcal{L}\right)$. The map
\[
(x,y) \longmapsto \left( 
\frac{y-x+r+s-1}{2},
\frac{-y-x+r+s+1}{2}
\right)
\]
transforms the set $\mathcal{L}$ into $\ZZ^2$, and $\mathcal{R}_+ \cap \mathcal{Z}$ into the set of solutions $(x,y)$ of the system:
\begin{equation}\label{eq:system}
\left\lbrace
\begin{array}{l}
x \geq \mu_2,\\
y \geq 0,\\
\mu_2+\nu_2-\lambda_2 \leq x+y \leq \mu_2+\nu_2-\lambda_3,\\
\lambda_2 \leq x-y \leq \lambda_2+\lambda_3
\end{array}
\right.
\end{equation}
This yields the following result. Note that an equivalent result is given by \citet{Thibon}.
\begin{proposition}
Let $\mu_2$, $\nu_2$, $\lambda_2$, $\lambda_3$ be nonnegative integers with $\lambda_2 \geq \lambda_3$. The Reduced Kronecker Coefficient $\overline{g}_{(\mu_2)(\nu_2)}^{(\lambda_2,\lambda_3)}$ counts the integral solutions $(x,y)$ of \eqref{eq:system}.
\end{proposition}
It follows a piecewise quasipolynomial description ${\bf D}_0$ for $\overline{g}_{(\mu_2)(\nu_2)}^{(\lambda_2,\lambda_3)}$. Its domain of quasipolynomiality are the $26$ maximal cells of a fan subdividing a cone of $\RR^4$. 

We now explain (ii). There exists a simple formula to recover Kronecker coefficients from reduced Kronecker coefficients.
\begin{theorem}[\citealp{Briand:Orellana:Rosas:Chamber}]
Let $\ell_1$ and $\ell_2$ be positive integers. Let 
$\lambda$, $\mu$, $\nu$ be partitions of the same weight such that $\ell(\mu)\leq \ell_1$, $\ell(\nu) \leq \ell_2$ and $\ell(\lambda) \leq \ell_1 \ell_2$. Then
\[
g_{\mu \nu}^{\lambda}=
\sum_{i=1}^{\ell_1 \ell_2-1} (-1)^{i+1}\; \overline{g}_{\cut{\mu}, \cut{\nu}}^{\lambda^{\dagger i}}
\]
where $\lambda^{\dagger i}$ is the partition obtained from $\lambda$ by incrementing the $i-1$ first terms and removing the $i$--th term, that is:
\[
\lambda^{\dagger i}=\left(1+\lambda_1,1+\lambda_2,\ldots,1+\lambda_{i-1},\lambda_{i+1},\lambda_{i+2}, \ldots\right)
\]
and $\cut{\lambda}=\lambda^{\dagger 1}$.
\end{theorem}
In particular, 
\begin{equation}\label{eq:from gbar}
g_{(\mu_1,\mu_2)(\nu_1,\nu_2)}^{(\lambda_1,\lambda_2,\lambda_3)}=
\overline{g}_{(\mu_2)(\nu_2)}^{(\lambda_2,\lambda_3)}
-
\overline{g}_{(\mu_2)(\nu_2)}^{(\lambda_1+1,\lambda_3)}
+
\overline{g}_{(\mu_2)(\nu_2)}^{(\lambda_1+1,\lambda_2+1)}.
\end{equation}
Formula \eqref{eq:from gbar} allows us to 
deduce from the description ${\bf D}_0$ a piecewise quasipolynomial description for $g_{(\mu_1,\mu_2)(\nu_1,\nu_2)}^{(\lambda_1,\lambda_2,\lambda_3)}$. Its domains of quasipolynomiality are the maximal cells of a polyhedral complex. But these cells are not cones. This is fixed to obtain the description {\bf D} where the domain of quasipolynomiality are cones. The key observation for this is: The quasipolynomials attached to contiguous maximal cells in ${\bf D}_0$ coincide not only on their common border, but also on close parallel hyperplanes.

\begin{remark}
It is also possible, in principle, to derive a piecewise quasipolynomial presentation for $g_{\mu,\nu}^{\lambda}$, $\ell(\mu) \leq m$, $\ell(\nu) \leq n$, $\ell(\lambda) \leq mn$, from Klimyk's formula, for any $m$, $n$. This is explained in \citetalias[][section 4.5]{GCT6}. Nevertheless, the domains of quasipolynomiality obtained this way are not cones. 
\end{remark}

\subsection{Checking {\bf Strong PH2} and {\bf Strong SH} for $g_{(\mu_1,\mu_2)(\nu_1,\nu_2)}^{(\lambda_1,\lambda_2,\lambda_3)}$}\label{subsec:get cex}

We explain how we check when the stretching function 
$\stretchg_{(\mu_1,\mu_2)(\nu_1,\nu_2)}^{(\lambda_1,\lambda_2,\lambda_3)}$ is positive and/or saturated from the description {\bf D}. 
For each of the $74$ maximal cells $\sigma$ in {\bf D}, the quasipolynomial formula on $\sigma$  has the form 
\begin{equation}\label{eq:shape}
g_{(\mu_1,\mu_2)(\nu_1,\nu_2)}^{(\lambda_1,\lambda_2,\lambda_3)}
=1/4\;Q_{\sigma}({\bf x})+1/2\;L_{\sigma}({\bf x})+\Delta_{\sigma}({\bf x})
\end{equation}
 where $Q_{\sigma}$ and $L_{\sigma}$ are
integral homogeneous polynomials in the variable ${\bf x}=(\lambda_1, \lambda_2, \lambda_3, \mu_2, \nu_2)$, respectively quadratic and linear. The function $\Delta_{\sigma}$ fulfills $\Delta_{\sigma}(0)=1$ and is constant on each coset of $\ZZ^5$ modulo the sublattice defined by $\mu_2+\nu_2 \equiv \lambda_1 \equiv \lambda_2 \equiv \lambda_3 \equiv 0 \mod 2$.

Checking positivity is specially easy because in \eqref{eq:shape} the functions $Q_{\sigma}$, $L_{\sigma}$, $\Delta_{\sigma}$ are nonnegative on $\sigma$ for nearly all cells $\sigma$. Indeed:
\begin{itemize}
\item For all $\sigma$ and all integral points ${\bf x} \in \sigma$, one has $Q_{\sigma}({\bf x})=\lim_{N\to\infty} g_{N\mu,N\nu}^{N\lambda}/N^2$, which is necessarily nonnegative.
\item We check by direct inspection that for all cells $\sigma$ and all ${\bf x}\in \sigma$, one has $L_{\sigma} ({\bf x})\geq 0$.
\item By direct inspection, for all cells $\sigma$ except four, $\Delta_{\sigma}({\bf x}) \geq 0$ for all ${\bf x} \in \sigma$. Let $A$ be the subset of $\ZZ^5$ defined by $\lambda_1 \equiv \lambda_2 \equiv \lambda_3 \equiv \mu_2+\nu_2+1 \equiv 0 \mod 2$.  For three of the four exceptional cells, call them $\sigma_1$, $\sigma_2$ and $\sigma_3$, one has, for all integral points ${\bf x} \in \sigma_i$, that $\Delta_{\sigma_i}({\bf x})=-1/4$ if ${\bf x}\in A$ and $\Delta_{\sigma_i} ({\bf x})\geq 0$ else. For the fourth exceptional cell, call it $\sigma_4$, one has, for all integral points ${\bf x} \in \sigma_4$, that  $\Delta_{\sigma_4}({\bf x})=-1/2$ if ${\bf x} \in A$ and $\Delta_{\sigma_4}({\bf x}) \geq 0$ else.
\end{itemize}
Therefore, the counter--examples to {\bf Strong PH2} correspond to ${\bf x} \in A \cap (\sigma_1 \cup \sigma_2 \cup \sigma_3 \cup \sigma_4)$.

Since {\bf Strong PH2} implies {\bf Strong SH}, we look for the counter--examples to {\bf Strong SH} among the counter--examples to {\bf Strong PH2}. 

For ${\bf x} \in A \cap \sigma_i$, $i \in \{1,2,3\}$, the coefficient $g_{\mu,\nu}^{\lambda}$ is zero if and only if 
$Q_{\sigma_i}({\bf x})/4+L_{\sigma_i}({\bf x})/2=1/4$, which is equivalent to $Q_{\sigma_i}({\bf x})-1=L_{\sigma_i}({\bf x})=0$ (because $Q_{\sigma_i}$ and $L_{\sigma_i}$ are integral). But $L_{\sigma_i}({\bf x})=0$ is found to be incompatible with ${\bf x} \in A$. Hence, the cells $\sigma_1$, $\sigma_2$, $\sigma_3$ provide no counter--example to {\bf Strong SH}.

For ${\bf x} \in A \cap \sigma_4$, 
the coefficient $g_{\mu,\nu}^{\lambda}$ is zero if and only if 
$Q_{\sigma_4}({\bf x})/4+L_{\sigma_4}({\bf x})/2=1/2$, which is equivalent to $Q_{\sigma_4}({\bf x})-2=L_{\sigma_4}({\bf x})=0$ or $Q_{\sigma_4}({\bf x})=L_{\sigma_4}({\bf x})-1=0$. The case $Q_{\sigma_4}({\bf x})-2=L_{\sigma_4}({\bf x})=0$ 
is discarded because $L_{\sigma_4}({\bf x})=|\lambda|-\mu_2-\nu_2$, whose vanishing is incompatible with ${\bf x}\in A$. On the contrary, $L_{\sigma_4}({\bf x})=1$ is compatible with ${\bf x} \in A$. Last $Q_{\sigma_4}$ factorizes as $(|\lambda|-2\;\mu_2)(|\lambda|-2\;\nu_2)$. The counter--examples to {\bf Strong SH} in $\sigma_4$ are thus given by ${\bf x}\in A$ and $|\lambda|=\mu_2+\nu_2+1=2\;\mu_2$. These are the counter--examples given in \ref{thm:main}.

\begin{remark} The reduced Kronecker coefficients $\overline{g}_{(\mu_2)(\nu_2)}^{(\lambda_2,\lambda_3)}$ fulfill Hypothesis {\bf Strong PH2}, and thus also satisfy  {\bf Strong SH}. \citep[See][]{Briand:Orellana:Rosas:Chamber}.
\end{remark}

\begin{remark}\label{rem:positivity index}
Let $\mu$, $\nu$ be partitions with at most two parts. Let $\lambda$ be a partition with at most three parts. Let ${\bf x}=(\lambda_1,\lambda_2,\lambda_3,\mu_2,\nu_2)$. Then
\[
\widetilde{g}_{\mu,\nu}^{\lambda}(N+1)=
N^2 /4\;Q_{\sigma}({\bf x})
+N/2\;\left(
L_{\sigma}({\bf x})+Q_{\sigma}({\bf x})\right)
+g_{\mu,\nu}^{\lambda}.
\]
Remember that for all cells $\sigma$ the polynomials $Q_{\sigma}$ and $L_{\sigma}$ are nonnegative on $\sigma$. This shows that for the Kronecker coefficients indexed by two two--row shapes, the \emph{positivity index} (see the appendix) is always at most $1$. 
\end{remark}


\section{On the complexity of computing the Kronecker Coefficients}\label{sec:relevance}

In this section we review the proof of the \compc{\#P}--hardness of \pb{Kron} given by \citet{Burgisser:Ikenmeyer}, and related results. Next we propose a new, very simple proof using known properties of the reduced Kronecker coefficients. This further underlines the relevance of the reduced Kronecker coefficients in computational complexity issues related to representation--theoretical structural constants.

The class \#P is a class of counting problems introduced by \citet{Valiant} in his study of the complexity of computing the permanent \citep[see also][Ch. 17]{Arora:Barak}. It consists of all functions $f: \{0,1\}^* \rightarrow \NN$ such that there exists a Turing machine $M$ working in polynomial time, and a polynomial $p$ such that for all $n \in \NN$ and all $x \in \{0,1\}^n$, 
\[
f(x)=\Card \,\{ y \in \{0,1\}^{p(n)}\,|\, M \text{ accepts } (x,y)\}.
\] 

A counting problem, corresponding to $g:\{0,1\}^* \rightarrow \NN$, is \#P--hard if for every function $f$ in \#P, there exists a polynomial--time Turing reduction from $f$ to $g$ (\emph{i.e.} $f$ is computed in polynomial time by a Turing Machine that has access to an oracle for $g$). The problem is \#P--complete if it is \#P--hard and belongs to \#P.
\emph{Parsimonious reductions} are a special kind of polynomial--time Turing reductions. Given two functions $f$, $g:\{0,1\}^* \rightarrow \NN$, one says that \emph{$f$ reduces parsimoniously to $g$} (notation: $f \Rp g$) if there exists a function $\varphi:\{0,1\}^* \rightarrow \{0,1\}^*$, computable in polynomial time, such that $f(x)=g(\varphi(x))$ for all $x \in \{0,1\}^*$.

Let $\lambda$ and $\mu$ be two partitions of length at most $n$. The \emph{Kostka number} $K_{\lambda,\mu}$ is the dimension of the weight space of weight $\mu$ in the irreducible representation $V_{\lambda}(GL_n(\CC))$.  Consider the following three problems:
\begin{Problems}
\item \pb{LRCoeff}: Given partitions $\lambda$, $\mu$, $\nu$ such that $|\lambda|=|\mu|+|\nu|$, compute the Littlewood--Richardson coefficient $c_{\mu,\nu}^{\lambda}$.
\item \pb{Kostka}: Given partitions $\lambda$ and $\mu$, compute the Kostka number $K_{\lambda,\mu}$.
\item \pb{KostkaSub}: Given partitions $\lambda$ and $\mu$, with $\ell(\lambda) \leq 2$, compute the Kostka number $K_{\lambda,\mu}$.
\end{Problems}
\citet{Narayanan} showed that \pb{LRCoeff} is \#P--complete
 as follows:
\begin{itemize}
\item \pb{LRCoeff} is in \#P because $c_{\mu,\nu}^{\lambda}$ counts the integral points of the Hive polytopes, which are described as the solution sets of systems of linear inequalities with size polynomial in the bitlength of $\lambda$, $\mu$, $\nu$ (that \pb{LRCoeff} is in \#P follows the same way from the Littlewood--Richardson rule, \citetalias[see][]{GCT3}).
\item He showed that $\pb{KostkaSub} \Rp \pb{Kostka} \Rp \pb{LRCoeff}$ and 
that a known \#P--complete problem (counting all contingency tables with two rows and prescribed column sums and row sums, see \citealp{Dyer:Kanna:Mount})
 reduces parsimoniously to $\pb{KostkaSub}$.
\end{itemize}

\citet{Burgisser:Ikenmeyer} showed that \textsc{Kron} is \#P--hard by exhibiting a reduction $\textsc{KostkaSub} \Rp \textsc{Kron}$ based on combinatorial constructions (\emph{Kronecker Tableaux}) due to \citet{Ballantine:Orellana}. Their proof, yet elementary, requires careful attention to detail. We propose an alternative, very simple, proof of the \#P--hardness of \pb{Kron}, relying on known properties of the reduced Kronecker coefficients. Consider the problem:
\begin{Problems}
\item \pb{RKron:} Given partitions $\lambda$, $\mu$, $\nu$, compute the reduced Kronecker coefficient $\overline{g}_{\mu,\nu}^{\lambda}$.
\end{Problems}
\noindent The \#P--hardness of \pb{Kron} is a consequence of the \#P--hardness of \pb{LRCoeff} and the existence of the following reductions:
\begin{itemize}
\item $\pb{LRCoeff} \Rp \pb{RKron}$. This follows from an observation by \citet{Murnaghan:1955} proved by \citet{Littlewood:1958}: For all partitions $\lambda$, $\mu$, $\nu$ such that $|\lambda|=|\mu|+|\nu|$, one has  $c_{\mu,\nu}^{\lambda}=\overline{g}_{\mu,\nu}^{\lambda}$.
\item $\pb{RKron} \Rp \pb{Kron}$. This follows from the existence of polynomial (even linear) bounds for the stabilization of the sequence of general term $g_{(n-|\alpha|,\alpha)(n-|\beta|,\beta)}^{(n-|\gamma|,\gamma)}$. For instance, by \citet{Vallejo}, 
\[
\overline{g}_{\alpha,\beta}^{\gamma}=g_{(n-|\alpha|,\alpha)(n-|\beta|,\beta)}^{(n-|\gamma|,\gamma)} 
\]
for all  $n \geq |\alpha|+|\beta|+\alpha_1+\beta_1+2\,|\gamma|$.
\end{itemize}

In the progression of computation problems associated to the structural constants of the representation theory of the groups $GL_n(\CC)$ and $\S_n$:
\[
\pb{Kostka} \Rp \pb{LRCoeff} \Rp \pb{RKron} \Rp \pb{Kron}
\]
(which can be continued further to the right with the structural constants of plethysm), 
the reduced Kronecker coefficients sit between the well--understood Littlewood--Richardson coefficients, and the still mysterious Kronecker coefficients. Understanding the reduced Kronecker coefficients may be a fruitful approach towards the understanding of the Kronecker coefficients. We introduced 
the reduced Kronecker coefficients in \ref{subsec:reduced} as limits of sequences of Kronecker coefficients, but they can be defined directly as the structural constants for the \emph{Character Polynomials} \citep[see][I.7, ex. 13 and ex. 14]{Macdonald}.


\newpage


\section*{Appendix:\\
 Erratum to the saturation hypothesis (SH)\\
in ``Geometric Complexity Theory VI'',\\
 by Ketan  Mulmuley}

Hypotheses {\bf SH}, {\bf PH2} and {\bf PH3} are corrected in this appendix  by appropriate relaxation without affecting the overall  approach of Geometric Complexity Theory (GCT).

Let $H=GL_n(\CC)\times GL_n(\CC)$ and 
$\rho: H \rightarrow G=GL(\CC^n \otimes \CC^n)=GL_{n^2}(\CC)$ the natural 
embedding. Let $\lambda$ and $\mu$ be partitions of length at most $n$,
and $\pi$ a partition of length at most $n^2$.
By a partition $\lambda$,  we mean a sequence $\lambda: \lambda_1 \ge 
\lambda_2 \ge \cdots \lambda_k >0$ of nonnegative integers, where
 $k$ is called the height or length of $\lambda$. Let $V_\lambda(GL_n(\CC))$
and $V_\mu(GL_n(\CC))$ be the Weyl modules of $GL_n(\CC)$ indexed 
by the partitions $\lambda$ and $\mu$, and $V_\pi(G)$ 
the Weyl module of $G$ indexed by $\pi$.
Let  $g_{\lambda,\mu}^\pi$ be the Kronecker coefficient. This
is the multiplicity of the 
$H$-module $V_\lambda(GL_n(\CC)) \otimes V_\mu(GL_n(\CC))$ in the $G$-module
$V_\pi(G)$, considered as an $H$-module via the embedding $\rho$.
Let $\tilde g_{\lambda,\mu}^\pi(n)=g_{n \lambda, n \mu}^{n \pi}$ be 
the associated stretching function, where $n$ is a positive integer.
It is shown in \citetalias{GCT6} that it is a quasipolynomial. 

Here we say that $f(n)$ is a quasipolynomial if there exist polynomials
$f_i(n)$, $1\le i \le l$, for some positive integer $l$, such that
$f(n)=f_i(n)$  if $n=i$ modulo $l$. We say it is {\em positive} 
if the coefficients of $f_i(n)$ are nonnegative for all $i$. We say it is
{\em saturated} if 
$f_i(n) > 0$ for every $n\ge 1$ whenever $f_i(n)$ is not identically zero--this
definition is slightly stronger than the one in \citetalias{GCT6}. Positivity
implies saturation.

It was conjectured in \citetalias{GCT6} that:
\begin{Problems}
\item {\bf (SH)}  \it{$\tilde g_{\lambda,\mu}^\pi(n)$ is  saturated.}
\end{Problems}
More strongly, 
\begin{Problems}
\item {\bf (PH2)} \it{$\tilde g_{\lambda,\mu}^\pi(n)$ is  positive.}
\end{Problems}
\ref{thm:main} and \ref{subsec:PH2} give counter examples to {\bf SH} and {\bf PH2} 
as stated. A similar phenomenon was also reported in \citetalias{GCT7,GCT8},
where it was observed that the structural constants of the 
nonstandard quantum groups associated with the  plethysm 
problem (of which the Kronecker problem is a special case) need not satisfy
an analogue of {\bf PH2}. But it was observed there that {\bf PH2} holds after 
a small adjustment. Specifically, define the {\em positivity index} 
$p(f)$ of a quasipolynomial $f(n)$ 
to be the smallest nonnegative integer such that
$f(n+p(f))$ is positive. The saturation index $s(f)$ is defined
 similarly.
As per the experimental evidence in \citetalias{GCT7,GCT8},
the positivity (and hence 
saturation) indices of the structural constants there are  small, though
not always zero; e.g. see Figures 30, 33, 35 in \citetalias{GCT8}. 
The same can be expected here. 
This is also supported by 
the experimental evidence for the Kronecker coefficients indexed by two two--row shapes, where too it may be observed that the positivity index is  small (see \ref{subsec:relaxed} and \ref{rem:positivity index}).

This leads to the following relaxed forms of {\bf SH} and {\bf PH2}.

\begin{hypothesis*}[{\bf SH}]

\noindent 
(a) The quasipolynomial  $k(n)=\tilde g_{\lambda,\mu}^\pi(n)$ is almost 
saturated; i.e. $s(k)=O(\poly(h))$, 
where $h \le n$ is the maximum of
the  heights  of $\lambda,\mu$ and $\pi$.
This means there exist nonnegative constants $a$ and $b$
(independent of $n$, $\lambda,\mu$ and $\pi$)
such that $s(k)\le a{h^b}$,

\noindent (b) The quasipolynomial
$\tilde g_{\lambda,\mu}^\pi(n)$ is strictly saturated, i.e. the
saturation index is zero,  for almost all $\lambda, \mu$, $\pi$. 
Specifically, the density of the triples $(\lambda,\mu,\pi)$ 
of total bit length $N$ with nonzero $g_{\lambda,\mu}^\pi$
for which the saturation index is not zero
is less than $1/N^{c''}$, for any positive constant $c''$, as 
$N\rightarrow \infty$. 
\end{hypothesis*}

The stronger hypothesis  than (a) is the following.

\begin{hypothesis*}[{\bf PH2}]

The quasipolynomial  $k(n)=\tilde g_{\lambda,\mu}^\pi(n)$ is almost
positive. This means $p(k)=O(\poly(h))$.
\end{hypothesis*}

Let 
\[K(t)=K_{\lambda,\mu}^\pi(t)=\sum_{n\ge 0} \tilde g_{\lambda,\mu}^\pi(n) t^n.\] 
It is shown in \citetalias{GCT6} that $K(t)$ is rational function
that can be expressed in a positive form:
\begin{equation} 
K(t)=\f{h_d t^d +\cdots + h_0}{\prod_{i=0}^k (1-t^{a_i})^{d_i}},
\end{equation} 
where (1) $h_0=1$, and $h_i$'s are nonnegative integers, (2) 
$a_i$'s and $d_i$'s are positive integers, 
(3)  $\sum_i d_i=d+1$, where $d$ is the degree of the quasipolynomial
$\tilde g_{\lambda,\mu}^\pi(n)$.
This positive form is not unique. But \citetalias{GCT6} also associates
with $g_{\lambda,\mu}^\pi$ a unique positive form that is minimal in 
a certain sense. 
Let $\chi_{\lambda,\mu}^\pi=\max\{a_i\}$, 
where $a_i$'s are the degrees occurring in this minimal form. We
call it the {\em modular index} of $g_{\lambda,\mu}^\pi$. 

The statement of {\bf PH3} in \citetalias{GCT6} should be replaced by
the following one.

\begin{hypothesis*}[{\bf PH3}]

The  modular index $\chi_{\lambda,\mu}^\pi=O(\poly(h))$.
\end{hypothesis*} 

This is stated as Conjecture~1.6.2  (for the plethysm problem) in
\citetalias{GCT6}. 
Here the minimal positive form is not conjectured  to be
reduced as in \citetalias{GCT6}. It can be shown that {\bf PH3} also implies {\bf SH} (a).

The article \citetalias{GCT6} also suggests an
approach  to prove a good polynomial bound on the modular index 
(i.e. {\bf PH3}) based on the theory of nonstandard quantum groups
\citetalias{GCT4,GCT7,GCT8}.
This is also a natural
approach to prove {\bf SH} (a) (and also {\bf PH2}) via {\bf PH3}.

As a preliminary result, it can be verified that
the hypothesis {\bf PH1} in \citetalias{GCT6}, or rather its slightly
strengthened form, implies 
weaker forms of  {\bf PH2} and {\bf SH}; specifically, that
$p(k)$ (and hence s(k)) is $O(2^{\text{poly}(h)})$. But 
this exponential  bound  is very weak and conservative.

The following is a relaxed  version of the Kronecker (decision) problem 
in \citetalias{GCT6}. 

\begin{problem*}[{\bf Relaxed Kronecker Problem}]
Given partitions $\lambda,\mu,\pi$ and a relaxation parameter 
$c > a{h^b}$, for the nonnegative constants 
$a$ and $b$ as in the statement of Hypothesis {\bf SH} above, decide whether 
 $g_{c \lambda, c \mu}^{c \pi}$ is positive.
\end{problem*} 

\begin{theorem*} 
Assuming {\bf PH1} in \citetalias{GCT6} and {\bf SH} here,
the relaxed Kronecker problem can be solved in
$\poly(\bitlength{\lambda},\bitlength{\mu},\bitlength{\pi},\bitlength{c})$
 time,
where $\bitlength{c}$ denotes the bitlength of $c$. In particular,
if the relaxation parameter $c$ is small \footnote {In this statement
it even suffices even if the bitlength of $c$, rather than its value,
is $O(\poly(h))$.},
i.e., $O(\poly(h))$,
then the time is polynomial in the bitlengths of $\lambda$, $\mu$ and $\pi$.

Furthermore, there exists a polynomial time algorithm for deciding
nonvanishing of $g_{\lambda,\mu}^\pi$ that works correctly on almost
all $\lambda,\mu$ and $\pi$; almost all has the same meaning here as in {\bf SH}.
\end{theorem*} 
This  follows by a  slight modification of the proof technique in \citetalias{GCT6}.

The hypotheses {\bf SH}, {\bf PH2} and {\bf PH3} associated with other  decision problems 
in \citetalias{GCT6} can be relaxed in a similar fashion, and analogues of the above Theorem can be proven for the relaxed versions of those decision
problems. {\bf SH} and {\bf PH2} in \citetalias{GCT8} can also be relaxed similarly.

These relaxations do not affect the overall GCT approach to the
fundamental lower bound
problems in complexity theory, such as the $P$ vs. $NP$ problem
in characteristic zero. Because the final goal in GCT is to use the
efficient algorithms for the decision problems
such as the Kronecker problem above,   or rather
the structure of these algorithms,   to show that, for every input size $n$,
there exists an obstruction $O_n$ that serves as a ``proof-certificate'' 
or  ``witness'' of hardness of the explicit function under consideration. 
When saturation only holds in a relaxed sense,
the bit length of the label of this final obstruction $O_n$ 
would get multiplied by a small (polynomial) factor. But this small
blow up does  not affect the overall approach.

The details of this relaxation 
would appear in the revised version of \citetalias{GCT6} under preparation.

\newpage

\begin{acknowledge}
Emmanuel Briand is supported by a Juan de la Cierva fellowship (MICINN, Spain). Mercedes Rosas is supported by a Ram\'on y Cajal fellowship (MICINN, Spain). Both are also supported by Projects MTM2007--64509 (MICINN, Spain) and FQM333 (Junta de Andalucia).
\end{acknowledge}

\bibliography{journals,../kroproduct}

\end{document}

\else


\typeout{Please call latex again.}
\makeatletter\expandafter\@@end
\fi